\documentclass[11pt]{amsart}

\usepackage[T1]{fontenc}

\usepackage{latexsym,enumerate}

\usepackage{mathrsfs}
\usepackage{amsmath,amssymb,amsthm,amsfonts,latexsym}
\usepackage{pstricks, pst-node, pst-text, pst-3d}
\usepackage[bookmarks]{hyperref}
\usepackage{srcltx}

\def\cal{\mathcal}

\def\adh#1{\overline{#1}}

\let\=\partial

\newtheorem {pro}{Proposition}[section]
\newtheorem {thm}[pro]{Theorem}
\newtheorem{lem}[pro]{Lemma}

\theoremstyle{definition}
 \newtheorem {rem}[pro]{Remark}
\newtheorem {dfn}[pro]{Definition}

\newtheorem {obs}[pro]{Observation}

\newcommand{\R}{\mathbb{R}}

\newcommand{\Cc}{\mathscr{C}}

\newcommand{\hn}{\mathcal{H}}

\newcommand{\bou}{ {\bf B}}

\newcommand{\ep}{\varepsilon}

\newcommand{\tra}{\mathbf{tr}}

\newcommand{\pa}{\partial}

\newcommand{\brzeg}{\partial^{1}_{reg}}
\newcommand{\skraj}{\partial}

\title[]{Uniform Poincar\'e inequality in o-minimal structures}

\makeatletter

\@namedef{subjclassname@2020}{%
  \textup{2020} Mathematics Subject Classification}

\makeatother

\author[A. Valette and G. Valette]{Anna Valette and Guillaume Valette}

\address[A. Valette]{Uniwersytet Jagiello\'nski, Katedra Teorii Optymalizacji i Sterowania, Wydzia\l\ Matematyki i Informatyki, ul. S. \L ojasiewicza 6, 30-348 Krak\'ow, Poland}
 \email{anna.valette@im.uj.edu.pl}
\address[G. Valette]{Uniwersytet Jagiello\'nski, Instytut Matematyki, ul. S. \L ojasiewicza 6, 30-348 Krak\'ow, Poland}
 \email{guillaume.valette@im.uj.edu.pl}


\thanks{Research partially supported by the Narodowe Centrum Nauki grant  2021/43/B/ST1/02359.}


\begin{document}
\begin{abstract}
We first define the trace on a domain $\Omega$ which is definable in an o-minimal structure. We then show that every function $u\in W^{1,p}(\Omega)$ vanishing on the boundary in the trace sense satisfies Poincar\'e inequality. We finally show, given a definable family of domains $(\Omega_t)_{t\in \R^k}$, that the constant of this inequality remains bounded,  if so does the volume of $\Omega_t$.
\end{abstract}
\maketitle
\section{Introduction}


Poincar\'e type inequalities are very valuable tools in the theory of PDE as well as in variational or numerical analysis. There are mainly two directions in the contemporary research: characterizing domains on which Poincar\'e type inequalities holds, see for instance \cite{mazya} or  \cite{ericsson} and the literature therein, or focus on the constant and find  uniform bounds for certain classes of domains. For the latter issue, parametric version of Poincar\'e inequalities are very useful \cite{boulk, unidom, ruiz, thomas}.

The aim of this note is to provide a proof of a parametric version of Poincar\'e-Friedrichs inequality in the o-minimal setting, see Theorem \ref{main}.
Definable sets can be considered as generalizations of  semialgebraic or subanalytic sets.   The theory of o-minimal structures, which  sits at the intersections of model theory, geometry, and analysis, is very adequate to perform analysis on singular sets. Thanks to their "tameness" the  o-minimal structures provide a sufficiently large playground for applicable finite dimensional variational and numerical analysis (especially in semialgebraic setting).
  
This is a part of our  project to carry out the theory of Sobolev spaces on definable domains, with possibly singular boundary. In foregoing papers, we proved a version of  Poincar\'e-Wirtinger inequality on subanalytic sets \cite{pi_vv}, studied  the trace operator on  $W^{1,p}(M)$ in the case where $p$ is large and $M$ is a bounded subanalytic submanifold of $\R^n$ \cite{trace}, and showed density of compactly supported functions in the  kernel of the trace operator for such Sobolev spaces.   Such a manifold $M$ may of course admit singularities in its closure which are not metrically conical. Some other interesting inequalities on singular subanalytic domains were investigated in \cite{bos1, bos2, thierry}. The advantage of working with o-minimal structures is that no extra ad hoc assumption on the metric geometry of the domains or on their boundaries is needed.

Some generalizations of Poincar\'e-Friedrichs inequality were given for example in \cite{pf, pf-w12}. Although the boundary of a domain which is definable in an o-minimal structure may admit singularities,
it must be an at least  $\Cc^1$ manifold almost everywhere.  We rely on it to  define a trace operator which is continuous in the $L^p_{loc}$ topology on the complement of a negligible subset of the boundary of the domain. This enables us to show a version of Poincar\'e-Friedrichs inequality which is valid on every definable domain $\Omega$ which is bounded in one direction for every element $u\in W^{1,p}(\Omega)$ that vanishes on the boundary in the trace sense.

 It is a very interesting problem to study the extent to which the best constant of Poincar\'e inequality is related to the geometry of the domain, especially  when singularities arise in its boundary. We show that there exists a uniform constant for every definable family of domains (i.e., definable with respect to parameters) of bounded volume.

We briefly recall that an \textbf{o-minimal structure} expanding the real  field $(\R, +,\cdot)$ is the data for every $n$ of  a Boolean algebra 
 $\mathcal{S}_n$  of subsets of $\R^n$   containing  all the algebraic subsets of $\R^n$ and satisfying
the following axioms:
\begin{enumerate}
\item If $A\in\mathcal{S}_m$, $B\in\mathcal{S}_n$, then $A\times B\in\mathcal{S}_{m+n}$; \item
If $\pi:\R^n\times \R\to \R^n$ is the natural projection and $A\in\mathcal{S}_{n+1}$, then $\pi(A)\in\mathcal{S}_n$;
 \item $\mathcal{S}_1$ is nothing else but all the finite unions of points and intervals.
\end{enumerate}
The elements of $\mathcal{S}_n$ are called the \textbf{definable}  subsets of ${\R}^n$.

Given a function $\xi:A\to \R$, with $A\subset \R^n$, $\Gamma_\xi$ will stand for the graph of $\xi$. If $\xi':A\to \R$ is another function, we set
$$(\xi,\xi'):=\{(x,y)\in A \times \R:\xi(x)<y < \xi'(x)\}.$$
 From now, we fix an o-minimal structure.

Let us recall the inductive definition of cells of $\R^n$. Every  subset of $\R^0=\{0\}$  is a cell. A definable subset $C$ is a cell of $\R^{n}$  if there is a cell $D$ of $\R^{n-1}$ such that one of the following conditions holds:
  \begin{enumerate}
   \item\label{item_type_1}  $C=\Gamma_{\xi}$ with  $\xi$ definable $\Cc^1$ function on $D$.
   \item\label{item_type_2}  $C=(\xi,\xi')$, where $\xi$ is either equal to $-\infty$ or  a $\Cc^1$   definable function on $D$, and  $\xi'$ is either equal $+\infty$ or a $\Cc^1$ definable   function  on $D$ satisfying $\xi< \xi'$.
  \end{enumerate}
A cell decomposition of $\R^n$ is a finite partition of $\R^n$ into cells which satisfies some extra inductive properties (see for instance \cite{costeomin,dries_omin}). We say that a cell decomposition is compatible with some given sets $A_1,\dots, A_m$ if each of the sets $A_i$ is a finite sum of cells of this cell decomposition. One of the main features of o-minimal structures is to always admit cell decompositions compatible with finitely many given definable sets.

\noindent{\bf Notations.} Given  $Z\subset\R^n$ and a function $u:Z\to \R$, as well as $p\in [1,\infty)$ and an integer $k$, we denote by $||u||_{L^p(Z,\hn^k)}$ the (possibly infinite) $L^p$ norm of $u$, where $\hn^k$ stands for the $k$-dimensional Hausdorff measure. As usual, we denote by $L^p(Z,\hn^k)$ the set of functions on $Z$ for which  $||u||_{L^p(Z,\hn^k)}$ is finite and,  by $L_{loc}^p(Z,\hn^k)$, we denote the space of functions $f:Z\to \R$ such that every point $x\in Z$ has a neighborhood $U$ in $Z$ for which the restriction $f_{|U}$ belongs to $L^p(U,\hn^k)$. We will then consider the $L^p_{loc}$- convergence on $Z$ in the following sense:  we say that a sequence of functions $f_i:Z\to\R$ converges to $f$ in  the $L^p_{loc}$- topology (and we denote it by $f_i\overset{L^p_{loc}}\longrightarrow f$) if every $x\in Z$ has a neighborhood $U$ in $ Z$ such that the restrictions ${f_i}_{|U}$ converge to $f_{|U}$ in the $L^p(U,\hn^k)$ norm.

 If  $\Omega\subset \R^n$ is an open set, we denote by $L^p(\Omega)$ the $L^p$ measurable functions on $\Omega$ (with respect to the Lebesgue measure) and by $||u||_{L^p(\Omega)}$ the $L^p$ norm.  We then let $$W^{1,p}(\Omega):= \{u\in L^p(\Omega),\; |\partial u| \in L^p(\Omega)\}$$ denote the Sobolev space, where $\pa u$ stands for the gradient of $u$ in the sense of distributions.
 It is well-known that this space, equipped with the norm
$$||u||_{W^{1,p}(\Omega)}:=||u||_{L^p(\Omega)}+| |\partial u||_{L^p(\Omega)}$$ is a Banach space, in which
  $\Cc^\infty(\Omega)$ is dense  for all $p\in [1,\infty)$.\\

 We put $\skraj\Omega:=\adh{\Omega}\setminus \Omega$ and we will denote by
 $\brzeg\Omega$ the set of points of $\skraj\Omega$ at which this set is a $\Cc^1$ submanifold of $\R^n$ of dimension  $(n-1)$.

We write  $|.|$  for the euclidean norm.  Given $x\in  \R^n$ and $\ep>0$, we denote by $\bou(x,\ep)$ 
the open ball of radius $\ep$ centered  at $x$ (for the euclidean norm) and by ${\bf S}^{n-1} $ the unit sphere (centered at the origin). The euclidean distance  of  a point $x\in\R^n$ to a subset $A\subset\R^n$ is denoted by $dist(x,A)$ and the canonical basis of $\R^n$ by $e_1,\dots, e_n$.  


\section{Trace operator}
In this section we  define a trace operator on the boundary of a given open definable set $\Omega\subset \R^n$. As definable sets are piecewise $\Cc^1$ submanifolds of $\R^n$, this actually follows from the classical theory.

For this purpose, we shall make use of the set $W^{1,p}_{loc}(\Omega \cup\brzeg\Omega)$,  which is defined as the  set constituted by all the distributions $u$ on $\Omega$ such that every  $x\in \Omega\cup\brzeg\Omega$ has a neighborhood $U$ in $\R^n$ such that  the restriction $u_{|U\cap \Omega}$ belongs to $W^{1,p}(U\cap \Omega)$. We then say that $(u_i)$ converges to $u\in W^{1,p}_{loc}(\Omega \cup\brzeg\Omega)$ if every $x\in \Omega\cup\brzeg\Omega$ has a neighborhood $U$ in $\R^n$ such that  the restriction $u_{i|U\cap \Omega}$ converges to $u_{|U\cap \Omega}$ in $W^{1,p}(U\cap \Omega)$. This defines the $W^{1,p}_{loc}$-topology.

\begin{dfn}\label{dfn_embedded}
We say that $\Omega$ is {\bf connected at $x\in \skraj\Omega$}\index{connected at $x$} if  $\bou(x,\ep)\cap \Omega$ is connected for all  $\ep>0$ small enough.
We say that $\Omega$ is {\bf normal} if it is connected at each $x\in \skraj\Omega$.
\end{dfn}
\begin{obs}\label{dense}If $\Omega$ is normal then
 $\Cc^1(\Omega\cup\brzeg\Omega)$ is dense in $W^{1,p}_{loc}(\Omega \cup\brzeg\Omega)$.
\end{obs}
\begin{proof}
Up to a partition of unity and a coordinate system of $\brzeg\Omega$, we are reduced to finding  a $\Cc^1$  approximation of 
a given function $u\in W^{1,p}(H_+)$, where $H_+=\{(x_1,\dots,x_n)\in \R^n: x_n> 0\}$ is the positive half-space, for which the result is well-known \cite[Section III.2.5]{bf}.\end{proof}
\begin{obs}\label{p2}
Let  $(f_i)\subset\Cc^1(\Omega\cup \brzeg\Omega) $ be a sequence.

\begin{enumerate}[(i)]
\item
If $(f_i)$ is converging  in the $W^{1,p}_{loc}$-topology, then $({f_i}_{|\brzeg\Omega})$ is also $L^p_{loc}$-converging.
\item If $(f_i)$ converges to $0$ in the $W^{1,p}_{loc}$-topology, then ${f_i}_{|\brzeg\Omega}\overset{L^p_{loc}}\longrightarrow 0$.
\end{enumerate}
\end{obs}
\begin{proof}
Fix $x\in\brzeg\Omega$. There are a  relatively compact neighborhood $U$ of $x$ in $\Omega\cup \brzeg\Omega$ and a constant $C>0$ such that  for all $\varphi\in\Cc^1(\Omega\cup \brzeg\Omega) $ the following inequality holds (for instance by \cite[Proposition III.2.18]{bf}) 
\begin{equation}\label{conv}
||\varphi||_{L^p(U \cap\brzeg\Omega,\hn^{n-1} )}\le C||\varphi||^{1-1/p}_{L^p(U\cap \Omega)}\;||\varphi||^{1/p}_{W^{1,p}(U\cap \Omega)}.
\end{equation}
Since the set $U$ is relatively compact in ${\Omega}\cup \brzeg\Omega$, we know that  if $(f_i)$ is convergent in $W^{1,p}_{loc}(\Omega\cup \brzeg\Omega)$ then $(f_{i | U\cap \Omega})$ is convergent in  $W^{1,p}(U\cap \Omega)$. 
We conclude from (\ref{conv}) that the sequence $({f_i}_{|U \cap \brzeg\Omega})$ is Cauchy in $ L^p(U\cap \brzeg\Omega,\hn^{n-1})$,
which  establishes $(i)$. This also yields  that if $(f_i)$ converges to $0$ in the $W^{1,p}_{loc}$-topology then ${f_i}_{|\brzeg\Omega}\overset{L^p_{loc}}\longrightarrow 0$.
\end{proof}

To define our trace operator in the case where $\Omega $ is normal 
for $\varphi\in\Cc^1(\Omega\cup\brzeg\Omega)$, let now $\tra(\varphi):=\varphi_{|\brzeg\Omega}$. By Observations \ref{dense} and \ref{p2}, this mapping extends to a mapping $$\tra_{\brzeg\Omega}:{W}^{1,p}(\Omega)\to L^p_{loc}(\brzeg\Omega,\hn^{n-1}),$$
which is continuous in the $W^{1,p}_{loc}(\Omega \cup\brzeg\Omega)$-topology.

Let us now define the trace in the case where $\Omega$ is not necessarily normal. There is a finite partition $\mathcal{P}$ of $\brzeg\Omega$ into definable sets such that every $(n-1)$-dimensional element $S$ of $\mathcal{P}$ is open in $\brzeg\Omega$ and the number of connected components of $\bou(x,\ep)\cap \Omega$  (for $\ep>0$ sufficiently small) is the same ($1$ or $2$) for all $x\in S$.   For instance, such a partition is given by the images of the open simplices of a $\Cc^0$ definable triangulation \cite{costeomin}. 
We start by defining $\tra_S$, where $S$ is an $(n-1)$-dimensional element of $\mathcal{P}$. 

\noindent {\it \underline {First case:} $\Omega$ is connected at the points of $S$.} In this case, there is a neighborhood $V$ of $S$ in $\R^n$ such that $V \cap \Omega$ is normal and we can define:
$$\tra_S:W^{1,p}(\Omega)\to L^p_{loc} (S,\hn^{n-1})^2,\quad u\mapsto (\tra_S \,u_{| V\cap \Omega},0).$$
Here, we add  the zero function as second component since there is only one connected component.  There will be a possibly nonzero component in the second case.

\noindent {\it \underline {Second case:} $\Omega$ fails to be connected at the points of $S$.} Let $V$ be a neighborhood of $S$ such that $V\cap\skraj\Omega=S$. If $V$ is sufficiently small, the set $V\cap\Omega$ has two connected components $\Omega_1$ and $\Omega_2$ such that $\Omega_1\cup S$ and $\Omega_2\cup S$ are both $\Cc^1$ manifolds with boundary. 

Let us denote by $\tra^1_S:W^{1,p}(\Omega_1)\to L^p_{loc}(S,\hn^{n-1})$ and $\tra_S^2:W^{1,p}(\Omega_2)\to L^p_{loc}(S,\hn^{n-1})$ the respective trace operators resulting from the above Observations. We then define $$\tra_S:W^{1,p}(\Omega)\to L^p_{loc} (S,\hn^{n-1})^2,\quad u\mapsto (\tra^1_S \,u_{| \Omega_1},\tra^2_S\, u_{|\Omega_2}).$$
Of course, this mapping depends on the way we have enumerated the connected components of $V\cap \Omega$, but, up to a possible permutation of the components, it is independent of any choice. In particular,  the  kernel of this mapping  is independent of our choices.

Finally, we define $\tra \,u$ by $\tra\, u(x)=\tra_S u (x)$, for $x\in S$, for each $S\in \mathcal{P}$ such that $\dim S=n-1$.   Since it is induced by the respective trace operators of the connected components of the germ of $\Omega$ at every point of   $x\in \brzeg\Omega$, it is an $L^p_{loc}$ mapping on this set.\\
 We let $$W^{1,p}(\Omega,\brzeg\Omega):= \ker\tra .$$

\begin{rem}\label{rem_trace}
It is easy to produce examples of definable domains $\Omega$ admitting a function $u\in W^{1,p}(\Omega)$ such that $\tra\, u$ is not $L^p$ on the boundary. Let for $k>2$, 
 $$\Omega_k:=\{(x,y)\in (0,1)^2:y < x^k\},$$ and let $u(x,y):=\frac{1}{x}.$ 
 Clearly,  
 $u\in W^{1,p}(\Omega_k)$, for $p\in [1,\frac{k}{2}]$, while $\tra\, u$ is not $L^p$ on $\brzeg\Omega_k$.  The results of \cite{trace} however yield that there is $p_0$ such that for $p\ge p_0$ and $u\in W^{1,p}(\Omega)$, $\tra\, u$ is always $L^p$, if $\Omega$ is a subanalytic bounded domain (the proof actually applies to any bounded domain which is definable in a polynomially bounded o-minimal structure expanding the real field). The real number $p_0$ depends on the Lipschitz geometry of $\skraj\Omega$.
\end{rem}

\section{Uniform bound for the Poincar\'e constants of definable families}
We say that $(A_t)_{t\in\R^k}$ is a {\bf definable family} if the set 
$$A:=\bigcup_{t\in\R^k}A_t\times \{t\}$$ is a definable subset of $\R^n\times\R^k$. 
We will sometimes regard a definable subset $A\subset\R^n\times\R^k$ as a definable family,
setting $$A_t:=\{x\in\R^n\,: (x,t)\in A\}.$$ 

Given two definable families $A\subset\R^n\times\R^k$ and $B\subset\R^m\times\R^k$, we say that $F_t:A_t \to B_t$, $t\in \R^k$, is a {\bf definable family of mappings} if the family of the graphs $(\Gamma_{F_t})_{t\in \R^k}$, is a definable family of sets of $\R^{n+m}$. 

A definable family of mappings $F_t:A_t\to B_t$, $t\in\R^k$ is {\bf uniformly Lipschitz (resp. bi-Lipschitz)} if there exists a constant $L$ such that
 $F_t$ is $L$-Lipschitz\footnote{i.e. $|F_t(x)-F_t(y)|\le L|x-y|$ for all $x,y\in A_t$}   (resp. $F_t$ is $L$-bi-Lipschitz) for all $t\in\R^k$. 

Here, we wish to emphasize that a definable family of sets (resp. mappings) is not only a family of definable sets (resp. mappings): all the fibers must glue together into a definable set (resp. mapping).  This confers to such families many uniform finiteness properties that are essential for our purpose (see \cite{costeomin, gv_livre}).

Given $X\in \mathcal{S}_n$, we denote by $(X)_{reg} $ the set of the points $x\in X$ at which  $X$ is a $\Cc^1$ submanifold of $\R^n$ (of any dimension) and by $T_x(X)_{reg}$ the tangent space to $(X)_{reg} $ at the point $x$.
We say that $\lambda \in {\bf S}^{n-1} $ is {\bf regular for $X$}  if there exists $\alpha >0$ such that for all $x\in (X)_{reg}$ :
\begin{equation}\label{eq_reg_familles}
  dist(\lambda, T_x (X)_{reg}) \geq \alpha.
 \end{equation}
We say that $\lambda$ is {\bf regular for a family $(A_t)_{t \in \R^k}$}, if there is $\alpha>0$ such that  (\ref{eq_reg_familles}) holds for $X=A_t$, for all  $t\in \R^k$  and all $x\in (A_t)_{reg}$. 

 If $\lambda\in {\bf S}^{n-1}$ is regular for $A\in {\cal S}_{n+k}$, it is of course regular for $A_t\in {\cal S}_n$ for every $t \in \R^k$, but it is indeed even stronger since 
 the  angle between the vector $\lambda$ and the tangent spaces to the fibers must then be bounded below away from zero by a positive constant {\it independent of the parameter $t$}.

Regular vectors do not always exist, even if the considered set has empty interior, as it is shown by the simple example of a circle. Nevertheless, when the considered sets have empty interior, up to a definable family of bi-Lipschitz maps, we can find such a vector:
\begin{thm}\label{proj_reg}
Let $(A_t)_{t\in \R^k}$ be a definable family of subsets of $\R^n$ such that $A_t$ has empty interior for each $t\in\R^k$. There exists a uniformly bi-Lipschitz definable family of homeomorphisms $h_t:\R^n\to\R^n$ such that the vector $e_n$ is regular for the family $(h_t(A_t))_{t\in \R^k}$. 
\end{thm}
 \begin{proof}It suffices to apply Theorem $3.1.3$ of \cite{gv_lip} to the generic fibers of the family $A$ (see \cite[section 5]{costeomin} for the definition of the generic fibers) and the compactness of  the Stone space of the Boolean algebra of the definable sets.     A more elementary explicit proof of this theorem, avoiding the abstract material of the Stone space, is provided in \cite[Theorem 2.2]{gv_reg} (see also \cite[Theorem 3.1.2]{gv_livre} for more details). \end{proof}

For   $\lambda\in{\bf S}^{n-1}$, the {\bf thickness of a set $\Omega\subset \R^n$ in the direction $\lambda$} is defined as  
$$|\Omega|_\lambda:=\sup\limits_{x\in\Omega}\;\sup\{s\ge 0\,:\, x+s\lambda\in\Omega\}.$$  
We say that $\Omega$ is {\bf bounded in the direction} $\lambda$ if $|\Omega|_\lambda<+\infty$.

We need the following  Poincar\'e type inequality on definable domains. 

\begin{thm}\label{p1}
Let $\Omega\subset\R^n$ be an  open definable subset bounded in the direction $\lambda\in {\bf S}^{n-1}$ and let  $u\in W^{1,p}(\Omega,\brzeg\Omega)$, $p\in [1,\infty)$.
We have
 \begin{equation}
||u||_{L^p(\Omega)}\le 2^{1/p}\; |\Omega|_\lambda||\partial u
||_{L^p(\Omega)}.
\end{equation} 
\end{thm}
\begin{proof}
Changing the coordinate system if necessary, we can assume that  $\lambda=e_n$.
Take a cell decomposition of $\R^n$ compatible with $\Omega$ and $\brzeg\Omega$,  and let $E$ be an open cell of this decomposition which is included in $\Omega$. This cell is of the form $\{(\tilde{x},x_n),\, \tilde{x}\in D, \, \xi_1(\tilde{x})<x_n<\xi_2(\tilde{x})\}$, where $\xi_1$ and  $\xi_2$ are $\Cc^1$ definable functions on an open cell $D\subset\R^{n-1}$ such that $\xi_1<\xi_2$.
If  the graph   $\Gamma_{\xi_1}$ (resp. $\Gamma_{\xi_2}$) is not in $\skraj\Omega$ then the cell which lies below (resp. above) $E$ is also included in $\Omega$. As we can make a bigger cell by gluing these cells to $E$, we may assume that  $\Gamma_{\xi_1}$ and $\Gamma_{\xi_2}$  are included in $\brzeg\Omega$. 

Define now for $\ep>0$ 
$$D_\ep:=\{x\in D\;:\;dist(x,\skraj D)>\ep\},$$
$$F_\ep:=\pi^{-1}(D_\ep)\cap (\xi_1,\xi_2),$$ where $\pi:\R^n\to\R^{n-1}$ is the projection onto the $(n-1)$ first coordinates. 
Since for each $\ep>0$ the set $\adh{F_\ep}$ is a compact subset of the manifold with boundary $E\cup \Gamma_{\xi_1}\cup \Gamma_{\xi_2}$ and the function $u$ has trace zero on the boundary, we can approximate $u$ by  smooth functions, i.e.   for every $\ep>0$ there exists a sequence  $(u_i)\subset\Cc^\infty_0(E)$ satisfying
\begin{equation}\label{eq_u}||u-u_i||_{W^{1,p}(F_\ep)}\to 0.\end{equation}
To construct such a sequence, we take local $\Cc^\infty_0$ approximations at the points of $\adh{F_\ep}$ (using for instance the mollyfying operators constructed in \cite[III.2.2]{bf})  that we then glue together  by means of a partition of unity subordinated to a finite covering of  $\adh{F_\ep}$.
Fix $\ep>0$ and set for simplicity $N:=|\Omega|_{e_n}$. By the Main Calculus Theorem, we have for $x=(x_1,\dots, x_n)\in E$:
\begin{equation}
u_i(x)=\int_{\xi_2(\tilde{x})-x_n}^{0}\frac{\partial u_i}{\partial x_n}(x_1,x_2,\dots, x_n+s)ds,
\end{equation}
and hence
\begin{equation}
|u_i(x)|\le\int_{0}^N\left|\frac{\partial u_i}{\partial x_n}(x_1,x_2,\dots, x_n+s)\right|ds.
\end{equation}
For $t\in\R$  and $\tilde{x}\in D$
we put 
\begin{equation}\label{eq_vs}
v_{i,t}(\tilde{x},s):=\begin{cases}\frac{\partial u_i}{\partial x_n}(\tilde{x},s+t), & (\tilde{x},s+t) \in F_\ep \\
0, &   (\tilde{x},s+t) \notin F_\ep .
\end{cases}
\end{equation}
 Then, by Minkowski's inequality 
\begin{equation}
||u_i||_{L^p(F_\ep)}= \left|\left|\int_{0}^{N}\left|v_{i,s}(\tilde{x},x_n)\right|ds\right|\right|_{L^p(F_\ep)}\le 
\int_{0}^{N}\left|\left|v_{i, s}\right|\right|_{L^p(F_\ep)}ds.
\end{equation}
 By H\"older's inequality we get
\begin{equation}
||u_i||_{L^p(F_\ep)}\le (N)^{1/p'}
\left(\int_{0}^{N}\left|\left|v_{i, s}\right|\right|_{L^p(F_\ep)}^p ds\right)^{\frac{1}{p}},
\end{equation}
with $\frac{1}{p}+\frac{1}{p'}=1$, 
which reads, since $v_{i,s}(x_n)=v_{i,x_n}(s)$,

$$
||u_i||_{L^p(F_\ep)} \le 
(N)^{1/p'}\left(\int_{0}^{N}\left(\int_{-N}^{N}\int_{D_\ep}\left|v_{i, x_n}(\tilde{x},s)\right|^p d\tilde{x}\ dx_n \right)ds\right)^{\frac{1}{p}}
$$
$$=(N)^{1/p'}\left(\int_{-N}^{N}\left(\int_{0}^{N}\int_{D_\ep}\left|v_{i, x_n}(\tilde{x},s)\right|^p d\tilde{x}\ ds \right )dx_n\right)^{\frac{1}{p}}$$
$$\le (N)^{1/p'}\left(\int_{-N}^{N} \left|\left|v_{i, x_n}\right|\right|^p_{L^p(\pi^{-1}(D_\ep))}dx_n\right)^{\frac{1}{p}}{=} 2^{1/p}N\left|\left|\frac{\partial u_i}{\partial x_n}\right|\right|_{L^p(F_\ep)},$$
since for all $t\in\R$ we have $|| v_{i,t}||_{L^p(\pi^{-1}(D_\ep))}\overset{(\ref{eq_vs})}{=}||\frac{\partial u_i}{\partial x_n}||_{L^p(F_\ep)}$.
Passing to the limit as $i\to\infty$ we get, thanks to (\ref{eq_u}),  that the above inequality holds for $u$ as well.
Making $\ep\to 0$ we finally obtain 
$$||u||_{L^p(E)} \le   2^{1/p}N||\partial u
||_{L^p(E)},$$
for every cell $E\subset \Omega$.
\end{proof}

Our uniform Poincar\'e inequality for definable families will require the following lemma.
 \begin{lem}
Let $\Omega\subset\R^n\times\R^k$ be a definable family of open sets. 
There exist $K>0$ and  a uniformly bi-Lipschitz definable family of  homeomorphisms $h_t:\R^n\to\R^n$   such that for any $t\in\R^k$:
\begin{equation}\label{lem1}|h_t(\Omega_t)|_{e_n}\le K \hn^n(\Omega_t)^{1/n}. \end{equation}
\end{lem}

\begin{proof}
It suffices to focus on the elements $t\in\R^k$ for which $\hn^n(\Omega_t)$ is finite since (\ref{lem1}) is trivial for the other parameters. As $\skraj\Omega_t$  is a definable family of sets of empty interiors, by Theorem \ref{proj_reg}, there is a uniformly bi-Lipschitz definable family of homeomorphisms $h_t:\R^n\to\R^n$ such that $e_n$ is regular for the family $(\skraj\Omega_t')_{t\in\R^k}$, where $\Omega'_t:=h_t(\Omega_t)$, which entails that  $\skraj\Omega'_t$  is comprised in the union of the respective graphs of some uniformly Lipschitz definable families of functions $\xi_{i,t}:\R^{n-1}\to\R$, $i=1,\dots, m$ (see for instance \cite[Proposition $3.6$]{gv_reg} or \cite[Chap. $3$]{gv_livre}).   Using the $
\min$  operator, we can transform these families into families that satisfy $\xi_{i,t}\le \xi_{i+1,t}$, for all $i< m$ and all $t\in \R^k$ (see \cite[Proposition $3.14$]{gv_reg}). 

Fix $t\in\R^k$. Since $\skraj\Omega_t'$ is included in the graphs of the functions $\xi_{i,t}$,  the set $\Omega_t'\setminus \bigcup_{i=1}^m \Gamma_{\xi_{i,t}}$ is the union of some connected components (which are finitely many) of the sets $(\xi_{i,t},\xi_{i+1,t})$ $0\le i\le m$, with $\xi_0\equiv -\infty$ and $\xi_{m+1}\equiv +\infty$.  As $\Omega'_t$ has finite volume, we therefore see that it actually must be included in $(\xi_{1,t},\xi_{m,t})$. 

Let us thus fix a connected component $E$ of $(\xi_{j,t},\xi_{j+1,t})\cap \Omega_t'$, for some $0<j<m$, and observe that, since the number of such connected components is bounded independently of $t$, it is enough to prove that we have \begin{equation*}\label{lemproof}|E|_{e_n}\le K \hn^n(\Omega_t')^{1/n}, \end{equation*}
for some constant $K$ independent of $t$.

We are going to show that this inequality is satisfied by any $K>4L^{1-1/n}$, where $L$ is the Lipschitz constant of the function $(\xi_{j,t}-\xi_{j+1,t})$. 
 Fix for this purpose a point $a$ in $\pi(E)$, where $\pi:\R^n\to \R^{n-1}$ is the projection onto the $(n-1)$ first coordinates, and
suppose by contradiction that 
$$\xi_{j+1,t}(a)-\xi_{j,t}(a)>K\hn^n(\Omega_t')^{1/n}.$$ Then, for any $x\in Z:=\bou(a,\frac{K \hn^n(\Omega_t')^{1/n}}{2L})\subset \R^{n-1}$ we have
$$\xi_{j+1,t}(x)-\xi_{j,t}(x)>\frac{K}{2}\hn^n(\Omega_t')^{1/n},$$  and hence, integrating over $Z$, we get (since $(\xi_{j,t|Z},\xi_{j+1,t|Z})\subset E\subset \Omega'_t$) 
$$\hn^n(\Omega_t')\ge\hn^{n-1}\left(\bou(a,\frac{K}{2L} \hn^n(\Omega'_t)^{1/n})\right)\cdot \frac{K}{2}\hn^n(\Omega_t')^{1/n}=\frac{K^n\,b_n}{2^nL^{n-1}}\hn^n(\Omega_t'),$$
where $b_n:=\hn^{n-1}(\bou(0_{\R^{n-1}},1))$.
If $K>4L^{1-1/n}$, we therefore have a contradiction (as $b_n\ge \frac{1}{2^{n}}$).
\end{proof}

We conclude, directly from the above lemma and Theorem \ref{p1}:
\begin{thm}\label{main}
For every definable family $\Omega\subset\R^n\times\R^k$ of open sets of finite volume there is a constant $C$
such that for all $t\in\R^k$ and all $u\in W^{1,p}(\Omega_t,\brzeg\Omega_t)$, $p\in [1,\infty)$,  we have:
 \begin{equation}
||u||_{L^p(\Omega_t)}\le C \hn^n(\Omega_t)^{1/n}||\partial u
||_{L^p(\Omega_t)}.\end{equation} 
\end{thm}

\end{document}